\documentclass[12pt,reqno]{amsart}
\usepackage[centertags]{amsmath}
\usepackage{amsfonts,amssymb}
\usepackage{mathrsfs}

\theoremstyle{plain} %default

\newtheorem{thm}{Theorem}

\theoremstyle{definition}

\theoremstyle{remark}

\numberwithin{equation}{section}

\DeclareMathSymbol{\R}{\mathalpha}{AMSb}{"52}
\DeclareMathSymbol{\C}{\mathalpha}{AMSb}{"43}

\newcommand{\comment}[1]{}

\newcommand{\fq}{\mathfrak{q}}

\newcommand{\beqn}{\begin{equation}}
\newcommand{\eeqn}{\end{equation}}
\newcommand{\balign}{\begin{align}}
\newcommand{\ealign}{\end{align}}
\newcommand{\bsube}{\begin{subequations}}
\newcommand{\esube}{\end{subequations}}

\begin{document}

\title[]{On  point transformations of linear equations of maximal symmetry}
\author[]{ JC Ndogmo}

\address{PO Box 728, Cresta 2118\\ South Africa
}
\email{ndogmoj@yahoo.com}

\begin{abstract}
 An effective method for generating  linear equations of maximal symmetry in their much general  normal form is obtained. In the said normal form, the coefficients of the equation are differential functions of the coefficient of the term of third highest order. As a result, an explicit expression for the point transformation reducing the equation to its canonical form is obtained, and a simple formula for the expression of the general solution in terms of those of the second-order source equation is recovered. New expressions for the general solution are also obtained, as well as a direct proof of the fact that a linear equation is iterative if and only if it is reducible to the canonical form by a point transformation. New classes of solvable equations parameterized by arbitrary functions are also derived, together with simple algebraic expressions for the corresponding general solution.
\end{abstract}

%
%%%% Subject Classification number
\maketitle

\section{Introduction}
\label{s:intro}
%%%%%%%%%%%%%%%%%%%%%%%%%%
%%%%%%%%%%%%%%%%%%%%%%%%
In a short paper published by Krause and Michel \cite{KM} in 1988 certain specific properties of linear equations of maximal symmetry were established. In particular, that paper shows that such equations are precisely the iterative ones, and equivalently those which can be reduced to the canonical form $y^{(n)}=0$ by a point transformation. However that short paper left a number of important questions unanswered. It does not provide for instance any expression for the point transformation mapping a given equation of maximal symmetry to the canonical form. On the other hand it also makes use without proof of a number of important results which were visibly already known but which are difficult to find in the literature. \par

Almost at the same time the problem of generation of equations of maximal symmetry was considered by Mahomed and Leach \cite{ML} who found an algorithm for obtaining expressions for the most general normal form of this equations based on the direct computation of the symmetry algebra. However, computations with this algorithm remain quite tedious and the authors managed to provide a general expression for the linear equations of maximal symmetry only up to the order eight.\par

A more direct algorithm for generating this class of equations based on the simple fact that they are iterative  was proposed recently in \cite{JF}. Nevertheless, some of the main results of that paper concerning in particular the generation and the point transformations of the class of linear equations of maximal symmetry still have room for improvements. \par

 In this paper, we provide a much simpler differential operator than that provided in \cite{JF} for generating linear iterative equations of a general order. As a result, we give a more direct proof than that of Krause and Michel \cite{KM} to the fact that linear equations reducible by an invertible point transformation to the form $y^{(n)}=0,$ and which will henceforth be refereed to as the canonical form,  are precisely those which are iterative. We also establish several results concerning the solutions of this class of equations, and in particular their transformation to canonical form. In contrast to the very well-known paper by Ermakov \cite{ermakov} who managed to find only some very specific cases of a very restricted class for which the second-order source equation is solvable, we provide full classes parameterized by arbitrary functions and for which the solution of the second-order source equation, and consequently of the entire class of equation of maximal symmetry, is given by  simple formulas.
 \comment{
 As usual in Lie theory, all the functions we shall considered will be assumed to be sufficiently smooth.
}

\section{Iterations of linear equations}
\label{s:iterat}
%%%%%%%%%%%%%%%%%%%%%
%%%%%%%%%%%%%%%%%%%%%
Let $r\neq0$ and $s$ be two smooth functions of $x,$  and consider the differential operator $\Psi = r \frac{d}{dx} + s.$ We shall often denote by $F[x_1, \dots, x_m]$ a differential function of the variable $x_1, \dots, x_m.$ Linear iterative equations are the iterations $\Psi^n [y]=0$ of the first-order linear ordinary differential equation $\Psi [y] \equiv r y'+ s y=0, $ given by
%%%
%%%
\beqn \label{iterdef}
\Psi^n [y] = \Psi^{n-1}(\Psi [y] ),\quad \text{ for $n\geq 1\quad$ with }\quad \Psi^0 =I,
\eeqn

where $I$ is the identity operator. A linear iterative equation of a general order $n$ thus has the form
\begin{equation}
\label{geniter}
\Psi^n [y] \equiv K_n^0\, y^{(n)} + K_{n}^1\, y^{(n-1)} +  K_{n}^2\, y^{(n-2)} + \dots +  K_{n}^{n-1}\, y' +  K_{n}^n\, y =0.
\end{equation}
Setting
\begin{equation}\label{knj=0}
K_m^j=0, \quad \text{for $j<0$ or $j > m$}, \text{ and }\quad   K_m^j=1, \quad \text{for $m=j=0$}
\end{equation}
and applying \eqref{iterdef} shows that the coefficients $K_n^j$ in the general expression \eqref{geniter} of an iterative equation satisfy the recurrence relation
\begin{equation} \label{gnlrec}
K_n^j= r K_{n-1}^{j} + \Psi K_{n-1}^{j-1}, \qquad \text{for $0 \leq j \leq n,$ and $n \geq 1$}.
\end{equation}
Moreover, setting $j=0$ or $j=n$ in \eqref{gnlrec} shows by induction on $n$ that
\begin{equation}\label{kn0knn}
K_{n}^0= r^n, \qquad K_n^n = \Psi^{n-1} [s], \qquad \text{for all $n \geq 1$},
\end{equation}
and applying \eqref{gnlrec} recursively and using the conventions set in \eqref{knj=0} give the new recurrence relation
\begin{equation} \label{recknj1}
K_n^j = \sum_{k=j}^n r^{n-k} \Psi K_{k-1}^{j-1}, \quad \text{for $j=0,\dots, n$ and $n \geq1.$}
\end{equation}
 We note that \eqref{recknj1} provides an algorithm for the computation of the coefficients $K_n^j$ in terms of the parameters $r$ and $s$ of the source equation and the operator $\Psi,$ and the resulting formula has effectively been obtained in \cite[Theorem 2.2]{JF}. Moreover, it is of course also possible to compute the $K_n^j$ directly in terms of $n$ and the parameters $r$ and $s,$ and for $j=1,2$ we find that
%%%%
%%%%
\begin{subequations}\label{kn12}
\begin{align}
K_{n}^1 &= r^{n-1} \left[ n s + \binom{n}{2} r' \right] \label{kn12a}  \\
K_n^2  &=  r^{n-2} \left[ \binom{n}{2} \Psi [s] + \binom{n}{3} \left( 3 s r' + r r''
    + \frac{3n-5}{4} r'^2 \right) \right]. \label{kn12b}
\end{align}
\end{subequations}
%%%
If we divide through the general $n$th order linear iterative equation $\Psi^n [y]=0$ in \eqref{geniter} by $K_n^0= r^n,$ it takes the form
\bsube \label{iterstd}
\begin{align}
0 &= y^{(n)} + B_n^1\, y^{(n-1)}+ \dots + B_n^j\, y^{(n-j)} + \dots + B_n^n\, y, \label{iterstd1} \\
B_n^j&= K_n^j / r^n.  \label{iterstd2}
\end{align}
\esube
It is clear that this equation represents the standard form of the general linear iterative equation with leading coefficient one. Moreover, the well known change of the dependent variable $y \mapsto y \exp\left(\frac{1}{n} \int_{x_0}^x B_n^1 (t)\, dt \right)$ maps \eqref{iterstd1} into its normal form in which the coefficient of the term of second highest order has vanished. This transformation however simply amounts to setting $B_n^1=0,$ i.e. $K_n^1 =0.$ Therefore, for given parameters $r$ and $s$ of the operator $\Psi,$ an $n$th order linear equation in normal form
\begin{subequations} \label{iternor}
\begin{align}
&y^{(n)} + A_n^2\, y^{(n-2)} + \dots + A_n^j\, y^{(n-j)} + \dots A_n^n\, y=0  \label{iternor1}\\
\intertext{is iterative if and only if }\vspace{-10mm}
& A_n^j= \frac{K_n^j}{r^n}\Bigg \vert_{K_n^1=0} , \qquad (2 \leq j \leq n), \label{iternor2}
\end{align}
\end{subequations}
where $K_n^j$ is given by ~\eqref{recknj1}. It follows from \eqref{kn12a} that the requirement that $K_n^1=0$ holds is equivalent to having
\begin{equation} \label{s(x)}
s = - \frac{1}{2} (n-1) r',
\end{equation}
and this shows in particular that any iterative equation in normal form can be expressed
in terms of the parameter $r$ alone, i.e. depends on a single arbitrary function. Clearly, the coefficients $A_n^j$ can also be expressed solely in terms of $n,$ $r$ and the derivatives of $r.$ For instance, by setting for any given function $\xi,$
\beqn \label{A(r)}
\mathcal{A}(\xi)(x) = \frac{[\xi'(x)]^2 - 2 \xi(x) \xi''(x) }{ 4 [\xi(x)]^2}
\eeqn
it follows from \eqref{kn12b} that in \eqref{iternor} we have
\beqn \label{an2}
A_2^2 = \mathcal{A}(r), \quad \text{ and more generally }\quad A_n^2 = \binom{n+1}{3} \mathcal{A}(r).
\eeqn
In fact, as already noted in \cite{ML, JF}, the coefficients $A_n^j$ depend only on the function $\mathcal{A}(r)= A_2^2$ and its derivatives. For simplicity, it will often be convenient to denote the coefficient $A_2^2$ of the term of third highest order in \eqref{iternor1} simply by $\fq.$ \par

Having noted that the coefficients of every linear equation of maximal symmetry in normal form depend solely on $\fq$ and its derivatives, an important problem considered in \cite{ML} was to find a linear ordinary differential operator $\Gamma_n [y] $ \emph{depending solely on} $\fq$ and its derivatives, and which generates  the most general form of the linear $n$th order equation of maximal  symmetry in the dependent variable $y=y(x).$  In a recent paper \cite{JF}, it was established that for an arbitrary parameter $r$ of the source equation, the operator
\begin{equation} \label{phin}
\Phi_n = \frac{1}{r^n} \Psi^n \Big \vert_{K_n^1 =0}
\end{equation}
generates the linear iterative equation of an arbitrary order $n$ in normal form and in its most general form \eqref{iternor}.  Therefore, although the operator \eqref{phin} and the equation it generates depend explicitly on $r$ (and not on $\fq = A_2^2$) and its derivatives, on the basis of the result of \cite{KM} identifying linear iterative equations with linear equations of maximal symmetry, the two equations  $\Gamma_n [y] =0$ and $ \Phi_n [y]=0$ should always be the same for all $n,$ although such an operator $\Gamma_n[y]$ has not yet been found. This is due in part to the fact that a general expression for $A_n^j= A_n^j[\fq]$ is not available for all $n$ and $j.$ Nonetheless, it is naturally possible to make use of the  differential operator $\Phi_n$ to generate directly a linear equation of maximal symmetry of the general form \eqref{iternor1} in which the coefficients $A_n^j$ depend only on $\fq$ and its derivatives.\par

Indeed, for any value of $n \geq 2$ it follows from the equations \eqref{recknj1}  and \eqref{iternor}-\eqref{an2} and a simple induction on $j \geq 2$ that each coefficient $A_n^j$ in \eqref{iternor} depends linearly on $r^{(j)}.$  Moreover, it follows from \eqref{A(r)}-\eqref{an2} that
\beqn \label{rulr(j)}
r^{(j)} = D_x^{j-2}  \left( \frac{r'^{\,2} - 4 \fq r^2}{2 r}\right):=F_j[r, \fq], \qquad \text{ for $j \geq 2,$}
\eeqn
where $D_x= d/dx.$ Therefore, applying the substitution \eqref{rulr(j)} to $\Phi_n [y]$  yields the desired equation, that is the linear equation of the form \eqref{iternor}, in which $A_n^j= A_n^j[\fq]$ depends only on $\fq$ and its derivatives.  More formally, the resulting differential operator can be represented as
\beqn \label{phin1}
\Phi_n^r = \frac{1}{r^n} \Psi^n \Big \vert_{K_n^1 =0, \;{\scriptsize \begin{cases}r^{(j)} = F_j[r, \fq]& \\ j \geq 2 &  \end{cases} }} = \Phi_n \Big \vert_{{\scriptsize \begin{cases}r^{(j)} = F_j[r, \fq]& \\ j \geq 2 &  \end{cases} }}.
\eeqn
For example, for $n=3$ or $4,$ evaluating  $\Phi_n [y]$ yields the following expressions directly in terms of $r$ and its derivatives alone.
\bsube \label{phin34}
\begin{align}
\Phi_3 [y] &= -\frac{y \left(r'^3-2 r r' r''+r^2 r^{(3)}\right)}{r^3}+ \frac{y' \left(r'^2-2 r r''\right)}{r^2}+y^{(3)} \\
\begin{split}\Phi_4 [y] &=  \frac{3 y \left(27 r'^4-68
r r'^2 r''+24 r^2 r' r^{(3)}+4 r^2 \left(7 r''^2-2 r r^{(4)}\right)\right)}{16 r^4}\\
 &\quad -\frac{5 y' \left(r'^3-2 r r' r''+r^2 r^{(3)}\right)}{r^3}+
\frac{5 \left(r'^2-2 r r''\right) y''}{2 r^2} +y^{(4)}.
\end{split}
\end{align}
\esube
However, if in addition we also apply to these expressions for $\Phi_n [y]$  the substitution \eqref{rulr(j)}, which amounts to applying directly the operator $\Phi_n^r$ to $y$ we obtain
\bsube \label{phin34r}
\begin{align}
\Phi_3^r [y]=&  2 \fq'y + 4 \fq y' + y''' \label{o3}\\
\Phi_4^r [y]=& 3 y \left(3 \fq^2+\fq''\right)+ 10 y' \fq'+10 \fq y''+y^{(4)} \label{o4}.
\end{align}
\esube
Comparing this with the known expressions for linear equations of maximal symmetry expressed solely in terms of $\fq$ and its derivatives \cite{ML, JF} shows that $\Phi_3^r$ and  and $\Phi_4^r$  yield indeed the indicated expressions.\par

If we set $r= u^2$ for a certain nonzero function  $u,$ the expression for $\mathcal{A}(r)= \mathcal{A}(u^2)$ in \eqref{A(r)} is much simpler and reduces to $-u''/u.$ Setting $\fq = \mathcal{A}(r)$ is thus equivalent to letting $u$ be a solution of the equation
\beqn \label{srce}
y'' + \fq y=0,
\eeqn

%% Wronskian
which is referred to as the second-order source equation for \eqref{iternor}. Thus, if we express $\Phi_n [y]$ with $r$ replaced by $u^2,$ the substitution rule for $u^{(j)}$ similar to that given for $r^{(j)}$ in \eqref{rulr(j)} takes the much simpler form
\beqn \label{ruluj}
u^{(j)} =  D_x^{j-2} \left( - u \fq\right):=H_j[u, \fq], \qquad \text{ for $j \geq 2.$}
\eeqn

Denote by $\Theta_n^u$ the operator $\Phi_n$ in which $r$ is replaced with $u^2$ and to which the substitution \eqref{ruluj} is applied. That is,

\beqn \label{phinu}
\Theta_n^u =\Phi_n \Big \vert_{{\scriptsize \begin{cases}r= u^2 &  \\ u^{(j)} = H_j[u, \fq],  &  j \geq 2 \end{cases} }}.
\eeqn
That is, $\Theta_n^u$ is just  $\Phi_n$ as in \eqref{phin}, in which the source parameter $r$ has been replaced by $u^2$ and to which the substitution \eqref{ruluj} is then applied.  Then for the same reasons that $\Phi_n^r[y, \fq]$ generates the most general form of linear iterative equations, $\Theta_n^u [y, \fq]$ also does the same. However, in view of the most simpler substitution rule \eqref{ruluj}, the computational cost for generating these equations is much lower using $\Theta_n^u$ rather than $\Phi_n^r.$ Indeed for arbitrary values of $\fq$, generating equations of maximal symmetry of order greater than ten has been up to now a very tedious task using the original algorithm of \cite{ML}, but it is now easy to generate such equations up to the order thirty in most standard computers using computing systems such as {\sc mathematica}. For instance, the coefficient $A_{15}^{15}$ of $y$ which is the largest expression in size in the equation $\Theta_{15}^u [y]=0$ is given by
\begin{equation*}
\begin{split}
A_{15}^{15}=& 14 \big(2 \big(52022476800 \fq^6 \fq'+2132810240 \fq'^5+6656237568 \fq^5 \fq^{(3)}\\
& +341232100 \fq'^3 \fq^{(4)} +1024 \fq^4\left(73853676 \fq' \fq''+286397  \fq^{(5)}\right) \\
&+32 \fq^3 \big(2646561024 \fq'^3+207959056 \fq'' \fq^{(3)}+123346720 \fq' \fq^{(4)} \\
&+184297 \fq^{(7)}\big) +2 \fq'^2 \left(863911980 \fq'' \fq^{(3)}+784597 \fq^{(7)}\right)\\
&+  8 \fq^2 \big(2094143648 \fq'^2 \fq^{(3)}+33615550 \fq^{(3)} \fq^{(4)} +22652990 \fq'' \fq^{(5)} \\
&+ \fq' \left(2835404512 \fq''^2+9750858 \fq^{(6)}\right)+7323 \fq^{(9)}\big)\\
&+13 \big(457296 [\fq^{(3)}]^3+\fq^{(5)} \left(740970 \fq''^2+2717 \fq^{(6)}\right)+2068 \fq^{(4)} \fq^{(7)}\\
&+2 \fq^{(3)} \left(1100540 \fq'' \fq^{(4)}+591 \fq^{(8)}\right)  +491 \fq'' \fq^{(9)}\big)+\fq' \big(780095196 \fq''^3\\
&+8478360 [\fq^{(4)}]^2+14231306 \fq^{(3)} \fq^{(5)}+8279146 \fq'' \fq^{(6)}+1793 \fq^{(10)}\big)\\
&+2 \fq \big(12696730880 \fq'^3 \fq''+151883460 \fq'^2 \fq^{(5)}+2519682 \fq^{(4)} \fq^{(5)}\\
&+2 \fq^{(3)} \left(257461378 \fq''^2+913529 \fq^{(6)}\right) +937436 \fq'' \fq^{(7)}+139 \fq^{(11)} \\
&+4 \fq' \big(95070193 [\fq^{(3)}]^2 +152580285 \fq'' \fq^{(4)}+79587 \fq^{(8)}\big)\big)\big)+\fq^{(13)}\big).
\end{split}
\end{equation*}

Using the operator $\Theta_n^u,$ we shall give in the next section a much simpler and direct proof than that of \cite{KM} that a linear equation is iterative if and only if it is  reducible by a point transformation to the canonical form. To close this section, we note that the Wronskian of any two linearly independent solutions $u$ and $v$ of the source equation \eqref{srce} is a nonzero constant and will be normalized to one. Nontrivial solutions of the second-order equation \eqref{srce} are not known for arbitrary values of the coefficient $\fq.$

%%%%%%%%%%%%%%%%%%%%%%%%%%%%%%%%%%%%%%%%%
%%%%%%%%%%%%%%%%%%%%%%%%%%%%%%%%%%%%%%%%%%%%%%%%%%%%%%%%%%%%%%%%%%%%%%%%%%%%%%%%%
\section{Point transformations}
\label{s:transfo}

%%%%%%%%%%%%%%%%%%%%%%%%%%%%%%%%%%%%%%%%%
%%%%%%%%%%%%%%%%%%%%%%%%%%%%%%%%%%%%%%%%%%%%%%%%%%%%%%%%%%%%%%%%%%%%%%%%%%%%%%%%%

The equivalence group of linear $n$th order equations in normal form \eqref{iternor1} is well known to be given by invertible point transformations of the form
\begin{equation}\label{eqvnor}
x = f(z), \qquad  y= \lambda [f'(x)]^{\frac{n-1}{2}} w,
\end{equation}
where $f$ \comment{f to \xi ?} is an arbitrary locally invertible function and $\lambda$  an arbitrary nonzero constant \cite{schw-eqv, jcnRW, jcnJAM}. In other words, an equation of the form \eqref{iternor1} is reducible to the canonical form if and only if there exist a transformation of the form \eqref{eqvnor} that maps such an equation to the canonical form.\par

 Let us denote by
\beqn \label{schwarzn}
S(\xi)(z)= \frac{( - 3 \xi''^2 + 2 \xi' \xi^{(3)})}{2\, \xi'{\,^2}}
\eeqn
the Schwarzian derivative of the function $\xi=\xi(z).$ By studying the expression of the source parameter of the transformed equation under equivalence transformations, a simple characterization of the point transformation that maps an iterative equation to the canonical form  was found in
\cite[Theorem4.3]{JF}. This result states that a point transformation reduces a given iterative equation, which without loss of generality may be assumed to be of the form  \eqref{iternor1}, to the canonical form $w^{(n)}(z)=0$ if and only if it is of the form \eqref{eqvnor}, where $f$ is the inverse of the function $z= h(x)$ satisfying
\begin{equation} \label{detAn2} A_2^2 (x) = \frac{1}{2} S (h)(x),\quad \text{ or equivalently, }\quad \mathcal{A}(r)(x) = \frac{1}{2} S (h) (x), \end{equation}
on account of \eqref{an2}, assuming that $r$ is the source parameter of the equation. By making use of this result of \cite{JF}, we can now provide a more direct proof than that given in \cite{KM} for the following result.

\begin{thm} \label{iter-reduc}
A linear equation is iterative if and only if it can be reduced to the canonical form by an invertible point transformation.
\end{thm}

\begin{proof}
As usual one may assume that the iterative equation is in its normal form and has source parameter $r.$ On the basis of the above stated result from Theorem 4.3 of \cite{JF}, it follows that the transformation \eqref{eqvnor} where $f$ is the inverse of the function $h= \int \frac{dx}{r}$ maps the iterative equation to its canonical form, as the latter expression for $h$ solves ~\eqref{detAn2}. More directly, one can prove that the class of iterative equations and that of equations reducible by an invertible point transformation to the canonical form are the same. Indeed, let $\Omega_n [f, \lambda, x, y]\equiv \Omega_n[f, \lambda]$ represent the element of the equivalence group of linear $n$th order equations $\Delta\equiv\Delta_n [y]=0$ in normal form, acting on the space of independent variable $x$ and dependent variable $y,$ and given by \eqref{eqvnor}. Denoting by $\Omega_n[f, \lambda] \cdot \Delta=0$ the transformed equation, it follows that we have exactly
\bsube \label{iter-transfR}
\begin{align}
\Phi_n [y] & =  K_n(\lambda, r)\; \Omega_n \left[\int \frac{dx}{r}, \lambda, z, w \right]\cdot w^{(n)} (z) \label{iter-transfR1}\\[-2mm]
\intertext{\vspace{-2mm} where}
 K_n(\lambda, r) &= 1/\left( \lambda [r(x)]^{(n+1)/2}\right).
\end{align}
\esube
In other words,  generating a linear iterative equation in $y=y(x)$ for any given value of $r$ using the differential operator $\Phi_n$ and transforming the canonical equation $w^{(n)}(z)=0$ in the new variables $x$ and $y$ using the point transformation operator $K_n(\lambda, r)\; \Omega_n \left[\int \frac{dx}{r}, \lambda, z, w \right]$ yield exactly the same equations, and this proves the result.
\end{proof}

In \cite{KM} Krause and Michel proved this theorem indirectly as a consequence of equivalence relations, by proving the equivalence between equations of maximal symmetry and iterative equations on one hand, and between equations reducible by invertible point transformations and equations of maximal symmetry on the other hand. The part in that proof stating that every equation reducible to the canonical form has maximal symmetry is due to Lie \cite{liecanonic}. The direct proof given here will allow us to construct various point transformations of practical importance for iterative equations as well as various forms of their general solutions.\par
%%%
%%%

It should first be noted that the right hand side of \eqref{iter-transfR1} gives another methods for generating linear iterative equations with source parameter $r.$  In practice however, linear iterative equations arise generally with no reference to any source parameter but are expressed solely in terms of the coefficient $\fq = A_2^2(x)$ and its derivatives. In this case, on the basis of a remark made in the previous section, a solution $h$ to the equation $\fq = \frac{1}{2}S(h)(x)$ in \eqref{detAn2} is simply given $h= \int \frac{d x}{u^2}$  where $u$ is a solution of the second-order source equation \eqref{srce}. Consequently, if we denote by $\Omega_n^u$ the transformation operator given by
%%%%
%%%%

\beqn \label{omega_n^u}
\Omega_n^u\left[\int \frac{dx}{u^2}, \lambda, z, w \right] = K_n(\lambda, u^2)\;  \Omega_n \left[\int \frac{dx}{u^2}, \lambda, z, w \right] \Big \vert_{{\scriptsize \begin{cases}u^{(j)} = H_j[u, \fq]& \\ j \geq 2 &  \end{cases} }},
\eeqn
where $H_j[u \fq]= D_x^{j-2}(-u \fq)$ as in \eqref{ruluj}, then $\Theta_n^u [y]$ and $\Omega_n^u \cdot w^{(n)}(z)$ generate exactly the same linear iterative equation expressed solely in terms of $\fq$ and its derivatives.\par

On the other hand, since the transformation operator $\Omega_n^u \left[\int \frac{dx}{u^2}, \lambda, z, w \right]$ maps the reduced equation to the most general form \eqref{iternor1} of the iterative  equation, its explicit expression can be used to derive the general solution of \eqref{iternor1} . Indeed, under this operator we have
\beqn \label{y-of-u}
y= \frac{1}{\lambda} u^{n-1} w,\qquad z= \int \frac{d x}{u^2}.
\eeqn
Consequently, $n$ linearly independent solutions to the general $n$th order iterative equation \eqref{iternor1} are given by
\beqn \label{sol-iter1}
   y_k = u^{n-1} \left( \int \frac{d x}{u^2} \right)^k,\qquad k=0, \dots, n-1.
\eeqn
In particular, if $u$ and $v$ are two linearly independent solutions of the source equation \eqref{srce}, then $\int \frac{d x}{u^2}= \frac{v}{u}.$ Consequently, in terms of $u$ and $v,$ the $n$ linearly independent solutions to \eqref{iternor1} in \eqref{sol-iter1} above can be rewritten as
\beqn \label{sol-iter2}
y_k= u^{n-1-k} v^{k}, \qquad k=0,\dots,n-1.
\eeqn

Formula \eqref{sol-iter2} is well known and was cited without proof in \cite{KM}, and it is an important result for which we have not been able to find the proof in the recent literature. Moreover, formula \eqref{sol-iter1} established above provides a much stronger result by showing that linearly independent solutions of \eqref{iternor1} can be expressed solely in terms of a single nonzero solution of the second-order source equation. Indeed, the  solutions $y_k$ in \eqref{sol-iter1} are clearly linearly independent as their Wronskian equals the nonzero constant $\prod_{j=1}^{n-1} j!.$

%% wronskian
%% normal form
\begin{thm} \label{characsol}
A linear equation is iterative if and only if it has $n$ linearly independent solutions $y_k$ of the form \eqref{sol-iter2}, where $u$ and $v$ are two linearly independent solutions of  the corresponding second-order source equation \eqref{srce}.
\end{thm}

\begin{proof}
The fact that a linear iterative equation has linearly independent solutions of the stated form is established in \eqref{sol-iter2}. Conversely, if a linear equation has $n$ linearly independent solutions of the stated form \eqref{sol-iter2}, then since the second-order source equation is reducible to the canonical form $w^{(n)}=0$ by a point transformation, without loss of generality we may assume that such a transformation reduces $u$ to 1 and $v$ to $z.$  Consequently, the corresponding solutions of the linear equation are polynomials of degree at most $n-1,$ and thus the transformed equation is in canonical form. It then clearly follows from Theorem \ref{iter-reduc} that the  equation is iterative.
\end{proof}

\section{Concluding remarks}
The results obtained in this paper show amongst others that linear equations of maximal symmetry are highly solvable, because that is indeed the case for the second-order source equation \eqref{srce} whose solutions completely determine those of the linear equations of maximal symmetry. In a very popular paper \cite{ermakov} published in Russian in 1880 and recently translated into English \cite{ermakovL}, Ermakov stated that the majority of second-order linear homogeneous ordinary differential equations for which it is possible to find conditions for their solvability are of the form
\beqn \label{ermakov-solv}
(\alpha_1 x^2 +  \alpha_2 x + \alpha_3 ) y'' + (\alpha_4 x + \alpha_5) y' + \alpha_6 y=0,
\eeqn
where the $\alpha_j$ are some arbitrary constants. He then moved on in the paper to obtain some very specific cases of equations which are solvable from this class, together with their general solutions. A couple of years later in the same decade Hill \cite{hill} considered in the study of lunar stability the most general form of the second-order linear equation in normal form, i.e. in the form \eqref{srce}, but in which the coefficient $\fq$ is a periodic function. Hill's equation as well as its important variants such as Meissner Equation and
Mathieu Equation have been been amply studied \cite{teschl}, and it is well known that its solutions and their related properties can be described by means of Floquet theory. This solutions can also be expressed in terms of Hill's determinant \cite{magnus}.\par

The results of this paper extend even more the scope of solvability of linear second-order equations obtained by Ermakov. Indeed, it follows from \eqref{sol-iter1} and the condition $r=u^2$ relating the parameter of the source equation and the solution of the corresponding second-order  equation, that for \emph{arbitrary} values $r=r (x)$ of the source parameter, two linearly independent solutions of the second-order equation of the form
\beqn \label{eqA(r)}
y'' + \mathcal{A}(r) y= 0
\eeqn
are given by $y_k= \sqrt{r} \left( \int \frac{d x}{r} \right)^k$ for $k=0, 1.$  In passing, we  note that $n$ linearly independent solutions of the corresponding $n$th iteration of \eqref{eqA(r)} can also be obtained from \eqref{sol-iter1} through the substitution $r= u^2.$  To make a much direct comparison with the class of equations \eqref{ermakov-solv} considered by Ermakov,  we note that linearly independent solutions can also be found for the standard form of \eqref{eqA(r)} which in terms of the arbitrary coefficient $B =B(x)$ of $y'$ takes the form
\beqn \label{stdform}
y'' + B y' + \frac{1}{4} (4 \mathcal{A}(r) + B^2 + 2 B')y = 0.
\eeqn
The two linearly independent solutions of \eqref{stdform} are given by
\beqn \label{stdform}
y_j= \sqrt{r} \left( \int \frac{d x}{r} \right)^j e^{-\frac{1}{2} \int B(x) dx}, \qquad j=0,1.
\eeqn
We note that as opposed to the very specific cases of solvable equations found by Ermakov from the much restricted class of equations \eqref{ermakov-solv}, not only \eqref{stdform} depends roughly speaking  on two arbitrary arbitrary functions, but also their solution is given by a very simple formula. One difficulty with the type of solutions found here for equations of maximal symmetry is however that in practice their coefficients are not generally given explicitly in terms if the source parameter $r,$ but solely in terms of $\fq$ and its derivatives, and solving $\mathcal{A}(r)=\fq$ for $r$ is equivalent to solving the corresponding source equation \eqref{srce} through the linearizing transformation $r=u^2.$

%%%%%%%%%%%%%%%%%%%%%%%%%%%%%%%%%%%%%%%%%%%%%%%%

%%%%%%%%%%%%%%%%%%%%%%%%%%%%%%%%%%%%%%%%%%%%%%%%%%%%

%%%%%%%%%%%%%%%%%%%%%%%%%%%%%%%%%%%%%%%%%%%%%%%%%%%%

%%%%%%%%%%%%%%%%%%%%%%%%%%%%%%%%%%%%%%%%%%%%%%%%

%%%%%%%%%%%%%%%%%%%%%%%%%%%%%%%%%%%%%%%%%%%%%%%%%%%%

%%%%%%%%%%%%%%%%%%%%%%%%%%%%%%%%%%%%%%%%%%%%%%%%

%%%%%%%%%%%%%%%%%%%%%%%%%%%%%%%%%%%%%%%%%%%%%%%%%%%%

\end{document}